\documentclass[11pt]{article}
\usepackage[margin=1in]{geometry}
\usepackage{amsmath,amssymb,amsthm}
\usepackage[T1]{fontenc}
\usepackage{lmodern}
\usepackage[hidelinks]{hyperref}
\newtheorem{theorem}{Theorem}
\newtheorem{lemma}[theorem]{Lemma}
\newtheorem{corollary}[theorem]{Corollary}
\theoremstyle{remark}
\newtheorem{remark}[theorem]{Remark}
\newcommand{\E}{\mathbb E}
\newcommand{\Pp}{\mathbb P}
\newcommand{\Bin}{\operatorname{Bin}}
\newcommand{\Fint}{\mathcal F^{\circ}}
\title{Reflection of optimal supports for\texorpdfstring{\\}{ }$k$-wise independent bits}
\author{Roy Hermann}
\date{September 7, 2026}
\begin{document}
\maketitle
\begin{abstract}
For even $k$, we derive a reflection identity for the basic count distributions in the problem of maximizing the probability that all $n$ $k$-wise independent Bernoulli variables equal one. After separating the distinguished node $n$, the other support nodes reflect by $s\mapsto n-1-s$ while the Bernoulli parameter changes from $p$ to $1-p$. Corresponding basic weights differ by an explicit positive factor. Thus interior feasibility sets, and local changes of optimal support with their multiplicities, are reflected. This gives a proof of the exceptional-point reflection in Conjecture 5.11 of Berend, Ernst, Kontorovich and Kumar. The argument uses Lagrange interpolation and an elementary binomial reweighting identity; it does not require a conjectured ordering or connectedness of the feasibility sets.
\end{abstract}

\section{The extremal problem and the basic distributions}

Let $M(n,k,p)$ be the largest possible value of
$\Pp(X_1=\cdots=X_n=1)$ when $X_1,\ldots,X_n$ are $k$-wise independent Bernoulli variables with common mean $p$.
Throughout, $k$ is even, $2\le k\le n-1$, and $0<p<1$.
Berend, Ernst, Kontorovich and Kumar~\cite{BEKK} study exact expressions for this maximum and conjecture a reflection of the parameter values at which several nodes of the optimal support change simultaneously.

Symmetrizing a feasible joint law preserves its objective and its $k$-wise independence. We may therefore work with the count $S=\sum_{i=1}^n X_i$ and its probabilities $w_s=\Pp(S=s)$, $0\le s\le n$. The constraints are
\begin{equation}\label{eq:moments}
 \sum_{s=0}^{n} w_s\binom{s}{r}=\binom{n}{r}p^r,
 \qquad 0\le r\le k,\qquad w_s\ge0.
\end{equation}
Conversely, a nonnegative solution of~\eqref{eq:moments} defines the required joint law by choosing uniformly among binary strings with $s$ ones conditional on $S=s$.
The objective is $w_n$. These $w_s$ are the total count probabilities, rather than the probabilities of individual binary strings.

The dual-feasible bases are characterized in~\cite[Theorem 3.3]{BEKK}. In the even-$k$ case they have the form
\begin{equation}\label{eq:basis}
 I=J\cup\{n\},\qquad
 J=\bigcup_{j=1}^{k/2}\{a_j,a_j+1\}\subseteq\{0,\ldots,n-1\},
\end{equation}
where $a_{j+1}\ge a_j+2$. For such an $I$, denote by $w_s^I(p)$ the unique signed solution of the moment equalities supported on $I$.
Uniqueness follows because the polynomials $\binom{x}{r}$, $0\le r\le k$, form a basis for polynomials of degree at most $k$, and evaluation on $k+1$ distinct nodes is invertible. We call $I$ feasible at $p$ if all its basic weights are nonnegative.

Define the involution
\begin{equation}\label{eq:star}
 J^*=\{n-1-s:s\in J\},\qquad I^*=J^*\cup\{n\}.
\end{equation}
It takes each pair $\{a,a+1\}$ to $\{n-a-2,n-a-1\}$ and preserves the class~\eqref{eq:basis}. The node $n$ is fixed separately.

\section{Reflection of the basic weights}

\begin{theorem}\label{thm:identity}
For every basis~\eqref{eq:basis}, every $s\in J$, and every $p\in(0,1)$,
\begin{equation}\label{eq:identity}
 w_{n-1-s}^{I^*}(1-p)
 =\frac{p(n-s)}{(1-p)(s+1)}\,w_s^I(p).
\end{equation}
Furthermore, $w_n^I(p)\ge p^n>0$.
In particular, if
\[
 \Fint(I)=\{p\in(0,1):w_s^I(p)\ge0\text{ for all }s\in I\},
\]
then $\Fint(I^*)=\{1-p:p\in\Fint(I)\}$.
\end{theorem}

\begin{proof}
For $s\in J$, let
\[
 \ell_s^J(x)=\prod_{t\in J\setminus\{s\}}\frac{x-t}{s-t}
\]
be the Lagrange polynomial on $J$. The Lagrange polynomial for the same node on $I$ is
\begin{equation}\label{eq:lagrange}
 L_s^I(x)=\frac{n-x}{n-s}\ell_s^J(x).
\end{equation}
Let $X\sim\Bin(n,p)$. The moment equalities imply exact agreement with the binomial law on every polynomial of degree at most $k$. Since $L_s^I$ has that degree and is one at $s$ and zero at the other nodes of $I$,
\begin{equation}\label{eq:expectation}
 w_s^I(p)=\E L_s^I(X).
\end{equation}
For any function $f$ on $\{0,\ldots,n\}$, the identity
$(n-x)\binom{n}{x}=n\binom{n-1}{x}$ gives
\begin{equation}\label{eq:reweight}
 \E[(n-X)f(X)]=n(1-p)\E f(Y),\qquad Y\sim\Bin(n-1,p).
\end{equation}
Consequently,
\begin{equation}\label{eq:reduced}
 w_s^I(p)=\frac{n(1-p)}{n-s}\E\ell_s^J(Y).
\end{equation}
Writing $s^*=n-1-s$, interpolation gives
\[
 \ell_{s^*}^{J^*}(x)=\ell_s^J(n-1-x).
\]
If $Y^*\sim\Bin(n-1,1-p)$, then $n-1-Y^*$ has the same law as $Y$. Applying~\eqref{eq:reduced} to $I^*$ at $1-p$ yields
\[
 w_{s^*}^{I^*}(1-p)
 =\frac{np}{s+1}\E\ell_{s^*}^{J^*}(Y^*)
 =\frac{np}{s+1}\E\ell_s^J(Y).
\]
Comparison with~\eqref{eq:reduced} proves~\eqref{eq:identity}, including at zeros of the weights.

The Lagrange polynomial for the distinguished node is
\begin{equation}\label{eq:dual}
 Q_I(x)=\prod_{j\in J}\frac{x-j}{n-j},\qquad
 w_n^I(p)=\E Q_I(X).
\end{equation}
For an integer $x$, every paired factor $(x-a_j)(x-a_j-1)$ is nonnegative, and every denominator $n-j$ is positive. Hence $Q_I(x)\ge0$ for $x\in\{0,\ldots,n\}$, with $Q_I(n)=1$. This proves $w_n^I(p)\ge p^n>0$.
All the multipliers in~\eqref{eq:identity} are positive on $(0,1)$, so the corresponding other weights have the same signs. Applying the same argument to $I^*$ proves the asserted equality of feasibility sets.
\end{proof}

\begin{lemma}\label{lem:optimal}
A feasible basis of the form~\eqref{eq:basis} gives the unique optimal count distribution.
\end{lemma}
\begin{proof}
The polynomial in~\eqref{eq:dual} satisfies $Q_I(s)\ge\mathbf 1_{\{s=n\}}$ on the integer support. For any feasible count distribution $\mu$, moment matching therefore gives
\[
 \mu_n\le\sum_{s=0}^n\mu_sQ_I(s)=\E Q_I(X)=w_n^I(p).
\]
A feasible basic distribution attains equality. Moreover, $Q_I(s)>0$ for $s\notin I$, so any distribution attaining equality is supported on $I$. The moment equations then imply uniqueness.
\end{proof}

\section{Exceptional points and local transitions}

For fixed $n,k$, each $w_s^I(p)$ is a polynomial in $p$ and is not identically zero. Indeed, its expression as $\E L_s^I(X)$ is its degree-$n$ Bernstein representation, whose coefficients are $L_s^I(x)$, $0\le x\le n$; these cannot all vanish since $L_s^I(s)=1$.
There are finitely many bases and finitely many roots of these polynomials in $(0,1)$.
Outside this finite set, all basic weights have fixed nonzero signs locally.

At every parameter, the finite linear program~\eqref{eq:moments} is feasible and bounded, and has an optimal basis. The dual-feasible basis characterization in~\cite[Theorem 3.3]{BEKK} supplies a feasible basis of the form~\eqref{eq:basis}. Outside the finite root set it has strictly positive weights, and Lemma~\ref{lem:optimal} makes this basis unique. Thus the unique optimal support is constant on each of the intervening open intervals.
Adjacent intervals with the same support may be merged. We call a boundary $d\in(0,1)$ a transition when the supports $I_-$ and $I_+$ immediately to its left and right differ. Its multiplicity is $|I_-\setminus I_+|=|I_+\setminus I_-|$; it is exceptional when this number exceeds one, as in~\cite[Section 5.3]{BEKK}.

\begin{corollary}\label{cor:transition}
If the transition at $d$ is $I_-\longrightarrow I_+$, then the transition at $1-d$ is $I_+^*\longrightarrow I_-^*$, with the same multiplicity. In particular, every exceptional point $d$ has an exceptional partner $1-d$.
\end{corollary}
\begin{proof}
Choose sufficiently small open intervals immediately on either side of $d$ that contain no further roots of basic weights. Theorem~\ref{thm:identity} sends their feasible bases to $I_-^*$ and $I_+^*$ on the reflected intervals. Reflection reverses their order. Their positive weights and Lemma~\ref{lem:optimal} identify them as the respective unique optimal supports. Since the map on nodes in~\eqref{eq:star}, with $n$ fixed, is a bijection,
\[
 |I_+^*\setminus I_-^*|=|I_+\setminus I_-|=|I_-\setminus I_+|.
\]
The last equality uses $|I_-|=|I_+|=k+1$.
\end{proof}

Corollary~\ref{cor:transition} establishes the reflection assertion of~\cite[Conjecture 5.11]{BEKK} in the even-$k$ setting of that paper's Section 5. Theorem~\ref{thm:identity} also reflects every interior feasibility component, whether or not the whole feasibility set is connected. We make no assertion about the conjectured lexicographic ordering of the bases or the full indexed formulation of~\cite[Conjecture 5.10]{BEKK}; our reflection is the explicit set map~\eqref{eq:star}.

\begin{remark}[Endpoints]
The restriction to $(0,1)$ is material. At $p=1$, every basis containing $n$ is feasible with the degenerate weight $w_n=1$, even if it is infeasible at every nearby interior parameter. These isolated endpoint solutions need not reflect to feasible solutions at $p=0$. The theorem applies to interior feasibility sets and, by continuity of reflection, their closures. It does not assert reflection of the literal full feasibility sets in $[0,1]$.
\end{remark}

\begin{remark}[Objective values]
The symmetry concerns the support structure and parameter locations. It does not assert $M(n,k,p)=M(n,k,1-p)$. If a feasible basic distribution at $p$ is known, its reflected weights away from $n$ are given by~\eqref{eq:identity}, and its remaining weight is determined by normalization.
\end{remark}

\section{Example and exact computational checks}

The exceptional pair $1/3,2/3$ for $n=11,k=4$ appears in~\cite[Table 3]{BEKK}. The adjacent bases are
\begin{align*}
 p=1/3:&\quad \{1,2,4,5,11\}\longrightarrow\{2,3,5,6,11\},\\
 p=2/3:&\quad \{4,5,7,8,11\}\longrightarrow\{5,6,8,9,11\}.
\end{align*}
Both transitions replace two nodes. The nonzero boundary count probabilities are
\begin{align*}
 p=1/3:&\quad (w_2,w_5,w_{11})=(110/243,44/81,1/243),\\
 p=2/3:&\quad (w_5,w_8,w_{11})=(22/81,55/81,4/81).
\end{align*}
Thus the two objective values differ, as expected.

The accompanying standard-library Python program uses exact rational arithmetic. It checks 1,390 bases: all bases for $3\le n\le14$ and even $2\le k\le\min(8,n-1)$, and 60 reproducible larger cases with $20\le n\le150$ and even $k\le16$. It verifies 8,510 instances of~\eqref{eq:identity} as polynomial identities, after clearing the denominator, and 9,900 moment identities coefficient by coefficient. It also checks nonnegativity of~\eqref{eq:dual} at each integer support point, the displayed boundary distributions, and an endpoint-degeneracy example. For each of the four one-sided supports in the example, the first nonzero Taylor coefficient of each weight at the boundary verifies its strict positivity on the indicated side. These finite checks supplement the general proof; they are not a formal proof-assistant verification or an independent expert review.

\section*{Disclosure of AI use}
This research was conducted through an interaction initiated and directed by Roy Hermann using OpenAI GPT-6 Astra in Codex. The AI system selected the research problem, generated the proposed proof, drafted the mathematical exposition, and wrote and executed the verification program. Its role was substantive mathematical work, not merely language editing. The conjecture and the prior extremal-probability framework are due to the authors cited below. No independent expert verification or peer review of this manuscript is claimed. The human author is responsible for the submitted work; the AI tool is not an author.

\end{document}